\documentclass[12pt]{article}

\usepackage{hyperref}
\usepackage{amsmath}
\usepackage{amssymb}
\usepackage{amsfonts}
\usepackage{amsopn}
\usepackage{amsthm}
\usepackage[all]{xy}

\swapnumbers
\theoremstyle{plain}
\newtheorem{thm}{Theorem}[section]

\newtheorem{lem}[thm]{Lemma}
\newtheorem{cor}[thm]{Corollary}
\newtheorem{prop}[thm]{Proposition}
\newtheorem{clm}[thm]{Claim}

\theoremstyle{definition}

\newtheorem{exs}[thm]{Examples}
\newtheorem{ex}[thm]{Example}
\newtheorem{rem}[thm]{Remark}
\newtheorem{dfn}[thm]{Definition}

\newcommand{\xra}[1]{\xrightarrow{#1}} 
\newcommand{\zz}{\mathbb{Z}} 
\newcommand{\nn}{\mathbb{N}}
\newcommand{\qq}{\mathbb{Q}} 

\newcommand{\A}{\mathrm{A}}
\newcommand{\B}{\mathrm{B}}
\newcommand{\D}{\mathrm{D}}
\newcommand{\E}{\mathrm{E}}
\newcommand{\F}{\mathrm{F}_4}
\newcommand{\pr}{\mathrm{pr}}
\newcommand{\res}{\mathrm{res}}
\newcommand{\rat}{\mathrm{rat}}
\newcommand{\BX}{\overline{X}}
\newcommand{\BG}{\overline{G}}
\newcommand{\rk}{\mathrm{rk}}
\newcommand{\ind}{\mathrm{ind}}
\newcommand{\pt}{\mathrm{pt}}
\newcommand{\an}{\mathrm{an}}
\newcommand{\w}{\bar{\omega}}

\DeclareMathOperator{\SB}{\mathrm{SB}}
\DeclareMathOperator{\Aut}{\mathrm{Aut}}
\DeclareMathOperator{\PGL}{\mathrm{PGL}}
\DeclareMathOperator{\SL}{\mathrm{SL}}
\DeclareMathOperator{\Mot}{\mathcal{M}}

\DeclareMathOperator{\Spec}{\mathrm{Spec}}
\DeclareMathOperator{\CH}{\mathrm{CH}}
\DeclareMathOperator{\Ch}{\mathrm{Ch}}
\DeclareMathOperator{\CHO}{\overline{\CH}}
\DeclareMathOperator{\ChO}{\overline{\Ch}}
\DeclareMathOperator{\Pic}{\mathrm{Pic}}

\title{Higher Tits indices of linear algebraic groups}
\author{Viktor Petrov\footnote{Supported partially by INTAS; project 03-51-3251.},
Nikita Semenov\footnote{Supported partially by DFG, project GI706/1-1.}}
\date{}

\begin{document}
\maketitle
\begin{abstract}
Let $G$ be a semisimple algebraic group over a field $k$.
We introduce the higher Tits indices of $G$ as the set of all
Tits indices of $G$ over all field extensions $K/k$.
In the context of quadratic forms this notion coincides with
the notion of the higher Witt indices introduced by M.~Knebusch and
classified by N.~Karpenko and A.~Vishik.
We classify the higher Tits indices for exceptional algebraic
groups. Our main tools include the Chow groups and the Chow motives
of projective homogeneous varieties, Steenrod operations, and
the notion of the $J$-invariant introduced in \cite{PSZ07}. 
\end{abstract}

\section{Introduction}
Let $G$ denote a semisimple algebraic group of inner type defined over a field $k$. 
In his famous paper \cite{Ti66} J.~Tits
defined the {\it Tits index} of $G$ as the data
consisting of the Dynkin diagram of $G$ with some vertices being circled.
Let $K$ be an arbitrary field extension of $k$. 
In the present paper we investigate the following problem:
What are the possible values of the Tits index of the group $G_K$? 

In the theory of central simple algebras (i.e., when
$G$ is a group of type $\A_n$) this problem is equivalent
to the index reduction formula of A.~Blanchet, A.~Schofield,
and M.~Van den Bergh (see \cite{SVB92}). Later their result was generalized
by A.~Merkurjev, I.~Panin, and A. Wadsworth (see \cite{MPW96} and \cite{MPW98}).

In the theory of quadratic forms (i.e, when $G$
is an orthogonal group) the above problem is equivalent to the
study of the {\it higher Witt indices} of quadratic forms.
The higher Witt indices were introduced by M.~Knebusch \cite{Kn76}.
They provide a nice discrete invariant of quadratic forms over any
field $k$. Numerous results have been obtained so far.
One of the main achievements here is the celebrated result
of N.~Karpenko \cite{Ka03}, where he proves Hoffmann's conjecture
about all possible values of the first higher Witt index.
One should also mention the interesting papers of N.~Karpenko, A.~Merkurjev
and A.~Vishik \cite{Ka04}, \cite{KM03}, \cite{Vi04}, \cite{Vi06}
concerning closely related problems in the theory of quadratic forms.

The main result of M.~Knebusch in his paper \cite{Kn76} asserts that one
can always consider only a certain finite number of field
extensions $K_i/k$, $i=0,\ldots,h$ such that the Tits index of $G$ over
a field $K/k$ equals one of the Tits indices of $G_{K_i}$.
The fields $K_i$ appearing in the Knebusch theory are the fields of
rational functions on certain projective $G$-homogeneous varieties
(see Section~\ref{tits} below).

Moreover, the result of Karpenko \cite[Theorem~2.6]{Ka04} asserts that
information on the higher Tits indices is hidden in a subring of the
Chow ring of certain projective $G$-homogeneous varieties\footnote{His result
concerns only quadrics, but can be straightforwardly generalized to arbitrary
projective homogeneous varieties.}.
In the present paper we exploit this connection further. Our main tools
include 
the Steenrod operations in the Chow theory, 
Tits' classification, Chow motives, and
motivic invariants, like the {\it $J$-invariant} of algebraic groups
introduced in \cite{PSZ07}.

It turns out, that in the most cases
the splitting behaviour of an algebaic group $G$ depends not
on the base field $k$, but on triviality or non-triviality
of a certain discrete invariant of $G$ called
the $J$-invariant. By definition this invariant measures the ``size''
of the subring of rational cycles in the Chow ring of the $G$-variety
of complete flags (see Section~\ref{secj}). But usually
the (non)-triviality of this invariant can be expressed in terms of the
(non)-triviality of the
Tits algebras of $G$ and/or of certain cohomological invariants of $G$,
like the Rost invariant.

In Section~\ref{sec6} we study the existence of the anisotropic kernels
of type $\D_6$ in the groups of type $\E_7$ over field extensions
$K/k$ of the base field (``index reduction formula'' for groups
of type $\E_7$). We prove that this problem is equivalent to the
problem of existence of zero-cycles of degree $1$ on certain anisotropic
projective homogeneous varieties. The latter problem has a long history
starting with the paper \cite{Serre} and obviously has a positive answer
(in the sense that an anisotropic projective variety does not have a zero-cycle
of degree $1$) over fields $k$ whose absolute Galois group
$\mathrm{Gal}(k_s/k)$ a pro $p$-group. We refer the reader to papers
\cite{Fl04}, \cite{Par05}, \cite{To04} that discuss this problem.
In the situation of $\E_7$ this problem was solved in article \cite{GS}
that was partially motivated by the results of the present paper. 

The main results of our paper are Theorems~\ref{thm1} and \ref{e7thm} that
allow to compute all possible higher
Tits indices for groups of type $\F$, $^1{\E_6}$, and $\E_7$ with trivial
Tits algebras, and to classify
all generically cellular varieties of exceptional types.

One suprising corollary of our results is a very short proof of the
triviality of the kernel of the Rost invariant for groups of type
$\E_7$ (see Corollary~\ref{coro2}). The arguments of that proof
belong to S.~Garibaldi.

The present paper contains an Appendix due to M.~Florence who shows
in a uniform way
that for all Dynkin types there exist a group whose set of higher
Tits indices is maximal possible (see Theorem~\ref{florence}).

The main goal of the present paper is to study general restrictions on the
splitting behaviour of algebraic groups, i.e., those restrictions
which don't depend on the base field.

\subparagraph{Acknowledgements} \mbox{}

We are sincerely grateful to S.~Garibaldi, Ph.~Gille, N.~Karpenko, and F.~Morel
for numerous interesting discussions on the subject of the present paper.

\section{Tits' classification and Knebusch theory}\label{tits}
We recall the definition and basic properties of the Tits indices of semisimple
algebraic groups following J.~Tits \cite{Ti66} and \cite{Ti71}.

Let $k$ be a field, $k_s$ be a separable closure of $k$, $G$ a semisimple
algebraic group defined over $k$, $S$ a maximal split torus of $G$
defined over $k$, $T$ a maximal torus containing $S$ and defined over $k$,
$\Delta=\Delta(G)$ the system of simple roots of $G$ with respect to $T$ and
$\Delta_0=\Delta_0(G)$ the subsystem of those roots which vanish on $S$.

In the present paper we consider only groups of {\it inner type}, i.e., groups
that are twisted forms of a split group by means of a $1$-cocycle
in $H_{\mathrm{et}}^1(k,G_0)$, where $G_0$ denotes the split adjoint group over $k$
of the same type as $G$. Equivalently, this means that the $*$-action
(see \cite{Ti66}) of the absolute Galois group on $\Delta$ is trivial.
Therefore we don't define the $*$-action and don't include it in the definition
of the Tits index of $G$.

The {\it index} of $G$ is a pair $(\Delta,\Delta_0)$. We represent the index
as the Dynkin diagram of $G$ with the vertices that don't belong to
$\Delta_0$ being circled.

There exists a certain subgroup of $G$
called the {\it semisimple anisotropic kernel}.
We refer the reader to \cite{Ti66} for its definition. Note that
the index of the semisimple anisotropic kernel of $G$ can be easiely
reduced from the index of $G$ by removing the vertices of the Dynkin diagram
which are circled.

To any semisimple group over $k$ one can functorially associate certain central
simple algebras, called the {\it Tits algebras}.
We refer the reader to \cite{Ti71} for a definition and description.

\begin{exs}
1. Let $A$ be a central simple $k$-algebra of degree $n+1$ and $G=\SL_1(A)$
the respective group of type $\A_n$. Then the index of $A$ equals
$\dfrac{n+1}{r+1}$, where $r$ is the number of circled vertices
on the Tits diagram of $G$.

The Tits algebras of $G$ are the $\lambda$-powers $\lambda^iA$, $i=1,\ldots,n$.

2. Let $(V,q)$ be a regular odd-dimensional quadratic space over $k$ and
$G=\mathrm{Spin}(V,q)$ be the respective group of type $\B_n$.
Then the number of circled vertices on its Tits diagram equals the Witt
index of $q$.

The Tits algebra of $G$ is the even Clifford algebra $C_0(V,q)$.
\end{exs}
One can give similar descriptions for all semisimple algebraic groups over $k$.

Next we recall the construction of the generic splitting tower of Knebusch
for semisimple algebraic groups (see \cite{Kn76}, \cite{Kn77}).

Consider the set
\begin{equation}\label{higher}
\{\ind(G_K)_{\an}\mid K/k\text{ is a field extension}\},
\end{equation}
where $\ind(G_K)_{\an}$ stands for the Tits index of the semisimple
anisotropic kernel of $G_K$.

\begin{dfn}
Set~(\ref{higher}) is called the set of the Higher Tits Indices of $G$.
\end{dfn}

This set can be obtained using the {\it generic splitting tower}, i.e.,
it suffices to consider not all field extensions $K/k$, but just
a finite number of {\it generic} ones. The latter are defined inductively
as follows.

First we set $K^0=k$, $G^1=G_{\an}$ and consider the function fields
$K^1_i=K^0(X_i)$, $i\in\Delta(G^1)\subset\Delta(G)$,
of the projective varieties of the maximal parabolic subgroups of $G^1$
of type $i$.
Note that there are precisely $\rk G_{\an}=|\Delta(G_{\an})|$ such varieties.

Next for each $j\in\Delta(G^1)$ we consider the group
$G^2_j=(G^1_{K^1_j})_{\an}$ over $K^1_j$ and
apply the same procedure, i.e., consider the fields
$K^2_i=K^1_j(X_i)$, $i\in\Delta(G^2_j)\subset\Delta(G)$, where $X_i$ stands
for the projective variety over $K^1_j$ of the maximal parabolic subgroups
of type $i$ of the group $G^2_j$. Proceeding further we obtain a set (a tower)
of fields $K_*^*$. Its main property is that
\begin{equation}\label{f1}
\{\ind(G_K)_{\an}\mid K/k\text{ is a field extension}\}
\end{equation}
\begin{equation*}
=\{\ind(G_{K^j_i})_{\an}\mid K^j_i/k\text{ is an element in the generic
splitting tower of }G\}.
\end{equation*}
The maximal value of the upper index of $K^*_*$'s is called the {\it height}
of $G$.

Conversely, given the set of the higher Tits index of $G$,
the Tits index of $G_{k(X_i)}$ can be restored as the minimal higher Tits
index containing $i$. An overview of this ideas can be found in
\cite{KR94}.

It is sometimes convenient to represent the right hand side of 
identity~(\ref{f1}) as an oriented labelled graph as following. First we define
a graph whose vertices are the anisotropic groups $G^s_j$
appearing in the construction above and which we represent by their Dynkin
diagrams.
There is an edge from $G^s_j$ to
$G^{s'}_{j'}$ with the label $i$ if and only if $s'=s+1$, $i=j'$ and
$G^{s+1}_{i}=((G^s_j)_{K^{s-1}_j(X_i)})_{\an}$.
Next we identify the vertices of this graph which correspond to
the same Dynkin types. Thus, the vertices of this graph represent all
possible Tits indices of $G_K$ without repetitions for all $K/k$.
The height of $G$ is the maximal length of the paths on the graph.

\begin{exs}
1. Let $G$ be an anisotropic group of type $\B_2$.
Then its splitting graph is:
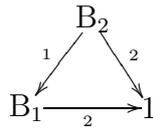
\xymatrix@M=0pt@=1em{
     & \B_2\ar[ldd]_{\text{{\tiny 1}}}\ar[rdd]^{\text{{\tiny 2}}} &\\
     &      &\\
\B_1\ar[rr]_{\text{{\tiny 2}}} &      & 1  
} (enumeration of simple roots follows Bourbaki).

2. Let $G$ be an anisotropic group of type $\B_3$ and $q$ the respective
quadratic form of discriminant $1$.
Then its splitting graph is: 
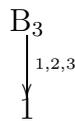
\xymatrix@M=0pt@=1em{
     \B_3\ar[dd]^{\text{{\tiny 1,2,3}}}\\
 \\
1  
} if $q$ has trivial Clifford invariant
and 
\xymatrix@M=0pt@=1em{
      &\B_3\ar[dddl]_{\text{{\tiny 1}}}\ar[dddr]^{\text{{\tiny 3}}}\ar[dd]_{\text{{\tiny 2}}}
      &  \\
      &                                  &\\
      & \B_1\ar[dr]_{\text{{\tiny 3}}}   &  \\
  \B_2\ar[rr]_{\text{{\tiny 3}}}\ar[ru]_{\text{{\tiny 2}}}  &   & 1 
}
otherwise.

There exist generalizations for quadratic forms of bigger
dimensions due to N.~Karpenko and A.~Vishik.
\end{exs}
The pictures in the examples resemble an automaton. Therefore one can
call the graphs defined above the {\it Tits automata}.

\begin{rem}
Note that the first step, i.e., the groups $G^2_*$ are the most important
ones. Indeed, the next groups $G^{\ge 3}_*$ involved in the construction
are anisotropic kernels of the groups $G^2_*$. Since the rank
of $G^2_*$ is smaller than the rank of $G^1$, we are deduced to the same
situation but for groups of a smaller rank.
\end{rem}

\section{Cycles on projective homogeneous varieties and
Chow motives}\label{sec2}
In this section we briefly describe the main properties of projective
homogeneous varieties and their Chow rings (see \cite{De74}, \cite{Hi82}).

Let $G$ be a split semisimple algebraic group of rank $n$ defined over a field $k$.
We fix a split maximal torus $T$ in $G$ and a Borel subgroup $B$ of $G$
containing $T$ and defined over $k$. 
We denote by $\Phi$ the root system of $G$,
by $\Pi=\{\alpha_1,\ldots,\alpha_n\}$ the set of simple roots of $\Phi$
with respect to $B$, by $W$ the Weyl group, and by $S=\{s_1,\ldots,s_n\}$
the corresponding set of fundamental reflections.

Let $P=P_\Theta$ be the (standard) parabolic subgroup corresponding
to a subset $\Theta\subset\Pi$, i.e., $P=BW_\Theta B$, where
$W_\Theta=\langle s_\theta, \theta\in\Theta\rangle$.
Denote $$W^\Theta=\{w\in W\mid\forall\, s\in\Theta\quad l(ws)=l(w)+1\},$$
where $l$ is the length function.
It is easy to see that $W^\Theta$ consists of all representatives
in the left cosets $W/W_\Theta$ which have minimal length.

As $P_i$ we denote the maximal parabolic subgroup $P_{\Pi\setminus\{\alpha_i\}}$
{\it of type $i$}
and as $w_0$ the longest element of $W$. Enumeration of simple roots follows Bourbaki.

Any projective $G$-homogeneous variety $X$
is isomorphic to $G/P_\Theta$ for some subset $\Theta$ of the simple roots.

Now consider the Chow ring of the variety $X=G/P_\Theta$. 
It is known that $\CH^*(G/P_\Theta)$ is a free abelian group
with a basis given by varieties $[X_w]$ that correspond 
to the elements $w\in W^{\Theta}$. 
The degree (codimension) of the basis element $[X_w]$ equals
$l(w_{\theta})-l(w)$, where $w_\theta$ is the longest element of $W_\Theta$.

Moreover, there exists a natural injective pull-back homomorphism
$$\CH^*(G/P)\to\CH^*(G/B)$$
$$[X_w]\mapsto[X_{ww_\theta}]$$

The following results provide tools to perform computations in the Chow
ring $\CH(G/P_\Theta)$.

In order to multiply two basis elements $h=[X_w]$ and $g=[X_{w'}]$ of $\CH^*(G/P_\Theta)$ such that 
$\deg h+\deg g=\dim G/P_\Theta$ we use the following
formula (Poincar\'e duality):
\begin{equation}\label{poinc}
[X_w]\cdot [X_{w'}]= \delta_{w,w_0w'w_\theta}\cdot [X_1].
\end{equation}
In view of Poincar\'e duality we denote as $[Z_w]$ the cycle dual
to $[X_w]$ with respect to the canonical basis. In other words,
$[Z_w]=[X_{w_0ww_{\theta}}]$.

In order to multiply two basis elements of $\CH^*(G/B)$ one of which
is of codimension $1$ we use the following formula 
(Pieri formula):
\begin{equation}\label{pieri}
[X_{w_0s_\alpha}][X_w]=
\sum_{\beta\in\Phi^+,\, l(ws_\beta)=l(w)-1}
\langle\beta^\vee,\w_\alpha\rangle[X_{ws_\beta}],
\end{equation}
where $\alpha$ is a simple root and the sum runs through
the set of positive roots $\beta\in\Phi^+$, $s_\beta$ denotes
the reflection corresponding to $\beta$
and $\w_\alpha$ is the fundamental weight corresponding
to $\alpha$. Here $[X_{w_0s_\alpha}]$ is the element of codimension $1$.

The {\it Poincar\'e polynomial} of 
a free abelian $\zz$-graded finitely generated group $A^*$ is, by definition, the
polynomial $g(A^*,t)=\sum_{i=-\infty}^{+\infty} a_it^i\in\zz[t,t^{-1}]$ with $a_i=\rk A^i(X)$.
The following formula (the Solomon theorem) allows to compute the Poincar\'e
polynomial of $\CH^*(X)$:
\begin{equation}
g(\CH^*(X),t)=\frac{r(\Pi)}{r(\Theta)},\ 
r(-)=\prod_{i=1}^l\frac{t^{d_i(-)}-1}{t-1},
\end{equation}
where $d_i(\Theta)$ (resp. $d_i(\Pi)$) denote the degrees of the fundamental polynomial invariants of
the root subsystem of $\Phi$ generated by $\Theta$ (resp. $\Pi$) and $l$ its rank
(see \cite{Ca72}). 
The dimension of $X$ equals $\deg g(\CH^*(X),t)$.
There exists a Maple package \cite{St} of J.~Stembridge
that provides tools to compute the Poincar\'e polynomials of projective
$G$-homogeneous varieties.

Let $\mathrm{P}=\mathrm{P}(\Phi)$ denote the weight space. 
We denote as $\w_1,\ldots\w_n$ the basis of $\mathrm{P}$
consisting of the fundamental weights. 
The symmetric algebra $S^*(\mathrm{P})$ is isomorphic to
$\mathbb{Z}[\w_1,\ldots\w_n]$. 
The Weyl group $W$ acts on $\mathrm{P}$, hence, on $S^*(\mathrm{P})$. 
Namely, for a simple root $\alpha_i$
$$s_i(\w_j)=
\begin{cases}
\w_i-\alpha_i, & i=j; \\
\w_j, & \text{otherwise}.
\end{cases}$$
We define a linear map $c\colon S^*(\mathrm{P})^{W_{\Theta}}\to\CH^*(G/P_\Theta)$
as follows.
For a homogeneous $W_{\Theta}$-invariant $u\in\mathbb{Z}[\w_1,\ldots,\w_n]$ 
$$
c(u)=\sum_{w\in W^\Theta,\, l(w)=\deg(u)}\Delta_w(u)[X_{w_0ww_\theta}],
$$ 
where for $w=s_{i_1}\ldots s_{i_k}$ we denote by $\Delta_w$
the composition of derivations
$\Delta_{s_{i_1}}\circ\ldots\circ\Delta_{s_{i_k}}$ and
the derivation $\Delta_{s_i}\colon S^*(\mathrm{P})\to S^{*-1}(\mathrm{P})$ is 
defined by $\Delta_{s_i}(u)=\dfrac{u-s_i(u)}{\alpha_i}$.

Let $U=\Sigma_u(P_\Theta)$ denote the set of the (positive) roots lying in the
unipotent radical of the parabolic subgroup $P_\Theta$.
Then the elementary symmetric polynomials
$\sum_{u\in U}\sigma_i(u)$ are $W_{P_\Theta}$-invariant and, in fact,
coincide with the Chern classes of the tangent bundle $T_X$:
\begin{equation}
c(T_X)=c(\prod_{\gamma\in U}(1+\gamma)).
\end{equation}
The Maple package \cite{map} provides efficient tools to compute
the Chern classes of the tangent bundles.

To multiply two cycles $[X_{w_1}]$ and $[X_{w_2}]$ in $\CH^*(X)$ we
proceed as follows. First, we find preimages of $[X_{w_1}]$ and $[X_{w_2}]$ in
$S^*(\mathrm{P})\otimes\qq=\qq[\w_1,\ldots,\w_n]$ (the preimages always exist;
see below), then 
we either expand the product in the polynomial ring
$\qq[\w_1,\ldots,\w_n]$ and apply the function $c$, or apply $c$ directly
using the Pieri formula~(\ref{pieri}), the Leibniz rule
\cite[Ch.~IV, Lemma~1.1(e)]{Hi82} or/and Poincar\'e duality~(\ref{poinc}).

To find a preimage of some $[X_w]$ we do the following. It is well known
that the map $c\otimes\qq\colon \qq[\w_1,\ldots\w_n]^{W_{\Theta}}\to
\CH^*(G/P_\Theta)\otimes\qq$ defined above is a ring epimorphism, and the ring
$\qq[\w_1,\ldots\w_n]^{W_{\Theta}}$ is generated by $\w_i$, $i\not\in\Theta$,
and by the $W_{\Theta}$-invariant fundamental polynomials for the 
semisimple part of the Levi part of $P_\Theta$,
i.e., for the split group of type $\langle\Theta\rangle\subset\Phi$.
The latter polynomials (as well as their degrees called
degrees of fundamental polynomial invariants) are known. Explicit formulas
for them are provided in \cite{Meh88}. Now, since we know a generating set of
$\qq[\w_1,\ldots\w_n]^{W_{\Theta}}$, we can compute its image in 
$\CH^*(G/P_\Theta)\otimes\qq$ and, thus, find a set of generators
of $\CH^*(G/P_\Theta)\otimes\qq$ together with their preimages
in $\qq[\w_1,\ldots\w_n]^{W_{\Theta}}$. Therefore we can compute a preimage
of any element in $\CH^*(G/P_\Theta)\otimes\qq$. Observe that we don't loose
any information extending scalars to $\qq$, since the group
$\CH^*(G/P_\Theta)$ is free abelian.

The effective procedures to multiply cycles 
in the Chow rings of projective homogeneous varieties
are implemented in the Maple package \cite{map}\footnote{Created
in collaboration with S.~Nikolenko and K.~Zainoulline.}.

Next we briefly describe Steenrod operations and motivic
decompositions of projective $G$-homogeneous varieties with isotropic group
$G$ following \cite{CGM05} and \cite{Br05}.
We refer the reader to the book \cite{EKM} of R.~Elman, N.~Karpenko,
and A.~Merkurjev or to the original paper \cite{Ma68} of Yu.~Manin
for the definition and properties of the Chow motives.

The main result of papers \cite{CGM05} and \cite{Br05} asserts that the Chow
motive a projective $G$-homogenous variety $X$ with isotropic group $G$
decomposes into a direct sum of (twisted) motives of anisotropic projective
$G_{\an}$-homogeneous varieties $Y_i$. Moreover, one has an explicit
algorithm to compute these motivic decompositions.

In the present section we give a combinatorial interpretation of these
decompositions in terms of the Hasse diagrams of the weak Bruhat order.
Namely, consider an oriented labelled graph, called {\it Hasse diagram} of $X$,
whose vertices are elements of $W^{\Theta}$, where $\Theta$ denotes
the type of the variety $X$, i.e., a graph whose vertices correspond
to the free additive generators of $\CH^*(\BX)$, where $\BX=X\times_{\Spec k}
\Spec k_s$ and $k_s$ stands for a separable closure of $k$.
There is an edge from a vertex $w$ to a vertex $w'$ labelled with $i$
if and only if $l(w)<l(w')$ and $w'=s_iw$.

Consider now the Chow motive of $X$ and erase from the Hasse diagram of
$X$ all edges with labels not in $\Delta_0$ (see Section~\ref{tits}). The Hasse
diagram splits then into several non-connected components which
correspond to the varieties $Y_i$. To illustrate this construction
we give the following example.
\begin{ex}\label{exe7}
Let $G$ be an isotropic group of type $\E_7$ such that $G_{\an}$
has type $\D_4$. This means that the vertices $1$, $6$, and $7$ on the
Tits diagram of $G$ are circled. Consider the projective $G$-homogeneous
variety $X$ of parabolic subgroups of type $7$. Its Hasse diagram is provided
in \cite[Figure~21]{PlSeVa}. Cutting the Hasse diagram along the edges
with labels $1$, $6$, and $7$ we see that the diagram splits
into $14$ components: $8$ alone standing vertices which correspond to
the elements of $W^\Theta$ of length $0$, $1$, $9$, $10$,
$17$, $18$, $26$, and $27$, and therefore to
the (twisted) Lefschetz motives $\zz$, $\zz(1)$, $\zz(9)$,
$\zz(10)$, $\zz(17)$, $\zz(18)$, $\zz(26)$, $\zz(27)$, and $6$
diagrams that correspond to different varieties of type $\D_4$.
It is well known and easy to see that $G_{\an}$ corresponds to
a(n anisotropic) $3$-fold Pfister form $\varphi$ and therefore by the celebrated result
of M.~Rost \cite{Ro98} the Chow motives of a projective $G_{\an}$-homogeneous
variety splits into a direct sum of (twisted) Rost motives $R$ which depend
only on $\varphi$.
The Rost motive $R$ is indecomposable and over $k_s$
(where $\varphi$ splits) $R_s\simeq\zz\oplus\zz(3)$.
Thus,
$$\Mot(X)\simeq
(\oplus_{i=0,1,9,10,17,18,26,27}\zz(i))\oplus
(\oplus_{i=2}^{22} R(i))\oplus R(11)\oplus R(12)\oplus R(13).$$
\end{ex}
One should note that in the category of the Chow motives with
finite coefficients of projective
homogeneous varieties the Krull-Schmidt theorem holds
(see \cite[Theorem~9.6]{CM06}). Therefore the motivic decompositions
are unique.

\medskip

Now we briefly recall
the basic properties of Steenrod operations constructed by V.~Voevodsky.
We follow P.~Brosnan \cite{Br03}

Let $X$ be a smooth projective variety over a field $k$ with
$\mathrm{char}\,k\ne 2$ and $p=2$.
For every $i\ge 0$ there exist certain homomorphisms $S^i=Sq^{2i}\colon
\Ch^*(X)\to\Ch^{*+i}(X)$ called Steenrod operations. The total
Steenrod operation is the sum $S=S_X=S^0+S^1+\ldots\colon\Ch^*(X)\to\Ch^*(X)$.
This map is a ring homomorphism. The restriction $S^i\vert_{\Ch^n(X)}$ is $0$
for $i>n$ and is the map $\alpha\mapsto\alpha^2$ for $n=i$. The map $S^0$
is the identity. Moreover, the total Steenrod operation commutes
with pull-backs and, in particular, preserves rationality of cycles.

To compute the Steenrod operations
on a projective $G$-homogeneous variety with a split group $G$ we use an
algorithm described in details in \cite{Du07}. This algorithm
is implemented in the Maple package \cite{map}.

\section{$J$-invariant}\label{secj}
In this section we recall the definition and the main properties of 
a motivic invariant of a semisimple algebraic group introduced in \cite{PSZ07}
and called the $J$-invariant. It was shown in \cite{PSZ07} that
this invariant determines the motivic behaviour of generically split
projective homogeneous varieties (see the definition below).

Let $G_0$ be a split semisimple algebraic group over $k$ with a 
split maximal torus $T$ and a Borel subgroup $B$ containing $T$.
Let $G={}_\gamma G_0$ be the twisted form of $G_0$ 
given by a $1$-cocycle $\gamma\in H^1(k,G_0)$.

Let $X$ be a projective $G$-homogeneous variety and $p$ a prime integer.
To simplify the notation we denote $\Ch^*(X)=\CH^*(X)\otimes\zz/p$
and $\overline X=X\times_{\Spec k}\Spec k_s$,
where $k_s$ stands for a separable closure of $k$.
We say that a cycle $J\in\CH^*(\overline X)$
(resp. $J\in\Ch^*(\overline X)$) is {\it rational}
if it lies in the image of the natural restriction map
$\res\colon\CH^*(X)\to\CH^*(\overline X)$
(resp. $\res\colon\Ch^*(X)\to\Ch^*(\overline X)$).
We denote as $\CHO^*(X)$ (resp. as $\ChO^*(X)$) the image of this map.
 
From now on and till the end of this section we consider
the variety $X={}_\gamma(G_0/B)$ of complete flags. 
Let $\widehat T$ denote the group of characters of $T$ and
$S(\widehat T)\subset S^*(\mathrm{P})$ be the symmetric algebra
(see Section~\ref{sec2}). 
By $R^*$ we denote the image of the characteristic map 
$c\colon S(\widehat T) \to \Ch^*(\BX)$ defined above.
According to \cite[Theorem~6.4]{KM05} $R^*\subseteq \ChO^*(X)$.

Let $\Ch^*(\BG)$ denote the Chow ring with $\zz/p$-coefficients
of the group $(G_0)_{k_s}$.
An explicit presentation of $\Ch^*(\BG)$ in terms of generators
and relations is known for all groups and all primes $p$.
Namely, by \cite[Theorem~3]{Kc85}
\begin{equation}\label{formkac}
\Ch^*(\BG)=(\zz/p)[x_1,\ldots,x_r]/(x_1^{p^{k_1}},\ldots,x_r^{p^{k_r}})
\end{equation}
for certain numbers $k_i$, $i=1,\ldots,r$, and $\deg x_i=d_i$
for certain numbers $1\le d_1\le\ldots\le d_r$ coprime to $p$.
A complete list of
numbers $\{d_ip^{k_i}\}_{i=1,\ldots,r}$, called {\it $p$-exceptional degrees}
of $G_0$, is provided in \cite[Table II]{Kc85}.  
Taking the $p$-primary and $p$-coprimary parts of each $p$-exceptional degree 
one immediately restores the respective $k_i$'s and $d_i$'s.

Now we introduce an order on the set of additive generators 
of $\Ch^*(\BG)$, i.e., on the monomials $x_1^{m_1}\ldots x_r^{m_r}$. 
To simplify the notation, we denote the monomial
$x_1^{m_1}\ldots x_r^{m_r}$ by $x^M$, where $M$ is an $r$-tuple
of integers $(m_1,\ldots,m_r)$. The codimension (in the Chow ring) of
$x^M$ is denoted by $|M|$. Observe that $|M|=\sum_{i=1}^rd_im_i$.

Given two $r$-tuples $M=(m_1,\ldots,m_r)$ and $N=(n_1,\ldots,n_r)$ we say
$x^M\le x^N$ (or equivalently $M\le N$) if either $|M|<|N|$, or $|M|=|N|$ and 
$m_i\le n_i$ for the greatest $i$ such that $m_i\ne n_i$.
This gives a well-ordering on the set of all monomials ($r$-tuples)
known also as {\it DegLex order}. 

Consider the pull-back induced by the quotient map
$$
\pi\colon \Ch^*(\BX)\to \Ch^*(\BG)
$$
According to \cite[Rem.~$2^\circ$]{Gr58} $\pi$ is surjective with
the kernel generated by
the subgroup of the non-constant elements of $R^*$.

Now we are ready to define the $J$-invariant of a group $G$.
\begin{dfn}\label{def71}
Let $X={}_\gamma(G_0/B)$ be the twisted
form of the variety of complete flags by means of a $1$-cocycle 
$\gamma\in H^1(k,G_0)$.
Denote as $\ChO^*(G)$ the image of the composite map
$$
\Ch^*(X)\xra{\res} \Ch^*(\BX) \xra{\pi} \Ch^*(\BG).
$$
Since both maps are ring homomorphisms, $\ChO^*(G)$ is a subring of $\Ch^*(\BG)$.

For each $1\le i\le r$ set $j_i$ to be the smallest non-negative
integer such that the subring $\ChO^*(G)$ contains an element $a$ 
with the greatest monomial $x_i^{p^{j_i}}$ 
with respect to the DegLex order on $\Ch^*(\BG)$, i.e.,
of the form 
$$
a=x_i^{p^{j_i}}+\sum_{x^M\lneq x_i^{p^{j_i}}} c_M x^M, \quad c_M\in\zz/p.
$$
The $r$-tuple of integers $(j_1,\ldots,j_r)$ is called
the {\it $J$-invariant of $G$ modulo $p$} and is denoted by $J_p(G)$.
Note that $j_i\le k_i$ for all $i$.
\end{dfn}

In case $p$ is not a torsion prime of $G$  we have $\Ch^*(\BG)=\zz/p$.
Therefore the $J$-invariant is interesting only for torsion primes
(see \cite[Definition~3]{Gr58} for a definition of torsion primes).
A table of possible values of the $J$-invariants is given
in \cite[Section~6]{PSZ07}.

To illustrate Definition~\ref{def71} of the $J$-invariant we give the following example.
For a prime integer $p$ we denote as $v_p$ the $p$-adic valuation.
\begin{ex}
Let $p$ be a prime integer and
$A$ and $B$ be central simple $k$-algebras that generate the same
subgroup in the Brauer group $\mathrm{Br}(k)$.
Set $G=\PGL_1(A)\times\PGL_1(B)$.

Then $J_p(G)=(v_p(\ind A),0)$.
Indeed, the Chow ring $$\Ch^*(\BG)=(\zz/p)[x_1,x_2]/(x_1^{p^{k_1}},x_2^{p^{k_2}})$$
with $k_1=v_p(\deg A)$, $k_2=v_p(\deg B)$. Therefore $r$ in the definition of the
$J$-invariant equals $2$. Denote $J_p(G)=(j_1,j_2)$ and
consider the map
$$\res\colon\Pic(X_A\times X_B)\to\Pic(\overline{X}_A\times\overline{X}_B),$$
where $X_A$ (resp. $X_B$) denote the $\PGL_1(A)$-
(resp. $\PGL_1(B)$-) variety of complete flags
and $\Pic$ stands for the Picard group modulo $p$.
Denote by $h_A$  (resp. $h_B$) the image of $\w_1\in S(\mathrm{P})$ in
$\Pic(\overline{X}_A)$ (resp. $\Pic(\overline{X}_B)$)
by means of the map $c$ defined in Section~\ref{sec2}. 

Since $A$ and $B$ generate the same subgroup in the Brauer group, the cycle
$1\times h_B+\alpha h_A\times 1\in\Pic(\overline{X}_A\times\overline{X}_B)$
is rational for some $\alpha\in(\zz/p)^\times$ (see \cite{MT95} for the
description of the Picard groups of projective homogeneous varieties).
The image of this cycle in $\Ch^*(\BG)$ by means of $\pi$
equals $x_2+\alpha x_1$ (at least
we can choose the generators $x_1$ and $x_2$ in such a way).
Therefore, since $x_1<x_2$
in the DegLex order, $j_2=0$. The proof that $j_1=v_p(\ind A)$ is the same
as in \cite[Section~7, case $\A_n$]{PSZ07} and we omit it.
\end{ex}

Next we describe some useful properties of the $J$-invariant.
\begin{prop}\label{lemj}
Let $G$ be a semisimple algebraic group of inner type over $k$,
$p$ a prime integer and $J_p(G)=(j_1,\ldots,j_r)$.
Then
\begin{enumerate}
\item Let $K/k$ be a field extension. 
Denote $J_p(G_K)=(j_1',\ldots,j_r')$.  Then $j_i'\le j_i$, $i=1,\ldots,r$.
\item 
Fix an $i=1,\ldots,r$.
Assume that in the presentation~(\ref{formkac}) 
for the semisimple anisotropic kernel $G_{\an}$ of $G$
none of $x_j$ has degree $d_i$. Then $j_i=0$.
\item Assume $d_i=1$ for some $i=1,\ldots,r$.

Then $j_i\le\underset{A}{\max}\, v_p(\ind A)$ where $A$ runs through all Tits algebras of $G$.
Conversely, if $j_i>0$, then there exists
a Tits algebra $A$ of $G$ with $v_p(\ind A)>0$.
\item Assume that the group $G$ does not have simple components
of type $\E_8$ and
for all primes $p$ the $J$-invariant $J_p(G)$ is trivial. Then $G$ is split.
\end{enumerate}
\end{prop}
\begin{proof}
1. This is an obvious consequences of the definition of the $J$-invariant.

2. 
Let $\Ch^*(G)=(\zz/p)[x_1,\ldots,x_r]/(x_1^{p^{k_1}},\ldots,x_r^{p^{k_r}})$
with $\deg x_i=d_i$,
$\Ch^*(G_{\an})=(\zz/p)[x'_1,\ldots,x'_{r'}]/(x_1^{p^{k'_1}},\ldots,
x_{r'}^{p^{k'_r}})$ with $\deg x'_i=d_i'$, and
$J_p(G_{\an})=(j'_1,\ldots,j'_{r'})$.
It follows from the \cite[Table~II]{Kc85} that
$r'\le r$ and $\{d'_i, i=1,\ldots,r'\}\subset\{d_i,i=1,\ldots,r\}$.
On the other hand, by \cite[Corollary~5.4]{PSZ07}, the polynomials
$\prod_{i=1}^r\dfrac{1-x^{d_ip^{j_i}}}{1-x^{d_i}}$ 
and $\prod_{i=1}^{r'}\dfrac{1-x^{d'_ip^{j'_i}}}{1-x^{d'_i}}$ are equal.
This implies the claim.

3.
Assume that $j_i>0$ and all Tits algebras of $G$ are trivial. Then
by \cite{MT95} the group $\Pic(\BX)$, where $X$ denotes the $G$-variety
of complete flags, is rational. Therefore by the very definition
of the $J$-invariant $j_i=0$. A contradiction.

Let $A$ be a Tits algebra of $G$ corresponding to a vertex $t$
of the Dynkin diagram such that $\pi(h_t)=x_i\in\Pic(\BG)$,
where $h_t\in\Pic(\BX)$
is the image of $\w_t\in S(\mathrm{P})$ by means of the map $c$ constructed above.
We show now that $j_i\le v_p(\ind A)=:s$, where $A$ is the Tits algebra
corresponding to the vertex $t$.

Consider the projective homogeneous variety $X\times\SB(A)$, where
$\SB(A)$ denotes the Severi-Brauer variety of right ideals
of $A$ of reduced dimension $1$.
Denote by $h_A\in\Pic(\overline{\SB(A)})=\Pic(\mathbb{P}^{\deg A-1})$ the
canonical generator as in Section~\ref{sec2}.

By the results of A.~Merkurjev and J.-P.~Tignol \cite{MT95} the cycle
$\alpha=h_t\times 1-1\times h_A\in\Pic(\BX\times\overline{\SB(A)})$ is rational.
Since the cycles $\alpha^{p^s}=h_t^{p^s}\times 1-
1\times h_A^{p^s}\in\Ch^*(\BX\times\overline{\SB(A)})$ and $h_A^{p^s}\in
\Ch^*(\overline{\SB(A)})$ are rational, the cycle $h_t^{p^s}\times 1\in
\Ch^*(\BX\times\overline{\SB(A)})$ is rational as well.

The projection morphism $\pr\colon X\times\SB(A)\to X$ is a projective bundle
by \cite[Corollary~3.4]{PSZ07}. In particular,
$\CH^*(X\times\SB(A))=\bigoplus_{j=0}^{\deg A-1}\CH^{*-j}(X)$.
Therefore the pull-back
$\pr^*$ has a section $\delta$.
By the construction of this section it is compatible with a base change.
Passing to the splitting field $k_s$
we obtain that
the cycle $\bar\delta (h_t^{p^s}\times 1)=h_t^{p^s}\in\Ch^*(\BX)$
is rational and the image $\pi(h_t^{p^s})=x_i^{p^s}$. By the definition
of the $J$-invariant, $j_i\le s$.

4. The statement follows from \cite[Corollary~6.10]{PSZ07} and
\cite[Theorem~C]{Gi97}.
\end{proof}

\begin{rem}
The fact that $j_i$ (with $d_i=1$)  provides an upper bound for $v_p(\ind A)$,
where $A$ runs through the Tits algebras of $G$, is not true.
A counter-example is e.g. a group of type $\E_7$ with a Tits
algebra of index more than $2$.
\end{rem}

\section{Generically split varieties}
In this section we begin to study the higher Tits indices of semisimple algebraic groups
over $k$. First, we would like to understand under what conditions
our group $G$ splits over the field of rational functions of a projective
$G$-homogeneous variety $X$.

\begin{dfn}
Let $G$ be a semisimple algebraic group over $k$ and $X$ a projective
$G$-homogeneous variety. We say that $X$ is {\it generically split}, if
the group $G$ splits (i.e., contains a split maximal torus) over $k(X)$.
\end{dfn}

\begin{rem}
If $X$ is generically split, then the Chow motive of $X$ splits
over $k(X)$ as a direct sum of Lefschetz motives.
This explains the terminology ``generically split''. One can also
call such varieties {\it generically cellular}, since over $k(X)$ they
are cellular via the Bruhat decomposition.
\end{rem}

\begin{thm}\label{thm1}
Let $G_0$ be a split semisimple algebraic group over $k$,
$G={}_\gamma G_0$ the twisted form of $G_0$ 
given by a $1$-cocycle $\gamma\in H^1(k,G_0)$,
and $X$ a projective $G$-homogeneous variety.
If $X$ is generically split, then for all primes $p$ the 
following identity on the Poincar\'e polynomials holds:
\begin{equation}\label{form}
\frac{g(\Ch^*(\BX),t)}{g(\ChO^*(X),t)}=\prod_{i=1}^r\dfrac{t^{d_ip^{j_i}}-1}{t^{d_i}-1},
\end{equation}
where $J_p(G)=(j_1,\ldots,j_r)$ and $d_i$'s are the $p$-coprimary
parts of the $p$-exceptional degrees of $G_0$. 
\end{thm}
\begin{proof}
In the proof of this theorem we use results established in our paper
\cite{PSZ07}.

Let $p$ be a prime integer. 
We fix preimages $e_i$ of $x_i\in\Ch^*(\BG)$. For an $r$-tuple
$M=(m_1,\ldots,m_r)$ we set $e^M=\prod_{i=1}^re_i^{m_i}$.

First we recall the definition of filtrations
on $\Ch^*(\BX)$ and $\ChO^*(X)$ (see \cite[Definition~4.13]{PSZ07}).
Given two pairs $(L,l)$ and $(M,m)$, where $L$ and $M$ are $r$-tuples
and $l$ and $m$ are integers, we say that $(L,l)\le(M,m)$ if
either $L<M$, or $L=M$ and $l\le m$.

The $(M,m)$-th term of the filtration
on $\Ch^*(\BX)$ is the subgroup of $\Ch^*(\BX)$ generated by the elements
$e^I\alpha$, $I\le M$, $\alpha\in R^{\le m}$. We denote as $A^{*,*}$
the graded ring associated to this filtration.  As
$A^{*,*}_{\rat}$ we denote the graded subring of $A^{*,*}$
associated to the subring $\ChO^*(X)\subset\Ch^*(\BX)$ of rational cycles with
the induced filtration.

Consider the Poincar\'e polynomial of $A_{\rat}$ with respect
to the grading induced by the usual grading of $\Ch^*(\BX)$.
Proposition~4.18 of \cite{PSZ07} which explicitely describes a $\zz/p$-basis
of $A^{*,*}_{\rat}$ implies that the Poincar\'e polynomial
$g(A_{\rat},t)=:\sum_{i=0}^{\dim X}a_it^i$ ($a_i\in\zz$) of
$A_{\rat}$ equals the right hand side of formula~(\ref{form}).

On the other hand, $\dim\ChO^*(X)=\dim A_{\rat}$ and 
the coefficients $b_i$ of the Poincar\'e polynomial
$g(\ChO^*(X),t)=:\sum_{i=0}^{\dim X}b_it^i$ are obviously bigger than or equal
to $a_i$ for all $i$. Therefore $g(\ChO^*(X),t)=g(A_{\rat},t)$.
This finishes the proof of the theorem.
\end{proof}

The right hand side of formula~(\ref{form}) depends only on the
value of the $J$-invariant of $G$. In turn, the left hand side
depends on the rationality of cycles on $X$. 
Available information on cycles that are rational as sure as fate,
allows to establish the following result.

\begin{thm}\label{cor1}
Let $G$ be a group of type $\Phi=\F$, $\E_6$, $\E_7$ or $\E_8$
given by a $1$-cocycle from $H^1(k,G_0)$, where $G_0$
stands for the split adjoint group of the same type as $G$,
and let $X$ be the variety of the parabolic subgroups of $G$ of type $i$.
The variety $X$ is not generically split if and only if

\begin{tabular}{|l|l|}
\hline
1&$\Phi=\F$,  $i=4$, $J_2(G)=(1)$\\
\hline
2&$\Phi=\E_6$,  $i=1,6$, $J_2(G)=(1)$\\
\hline
3&$\Phi=\E_6$,  $i=2,4$, $J_3(G)=(j_1,*)$, $j_1\ne 0$\\
\hline
4&$\Phi=\E_7$,  $i=1,3,4,6$, $J_2(G)=(j_1,*,*,*)$, $j_1\ne 0$\\
\hline
5&$\Phi=\E_7$,  $i=1,6,7$, $J_2(G)=(*,j_2,*,*)$, $j_2\ne 0$\\
\hline
6&$\Phi=\E_7$, $i=7$, $J_3(G)=(1)$\\
\hline
7&$\Phi=\E_8$, $i=1,6,7,8$, $J_2(G)=(j_1,*,*,*)$, $j_1\ne 0$\\
\hline
8&$\Phi=\E_8$, $i=7,8$, $J_3(G)=(1,*)$\\
\hline
\end{tabular}

{\footnotesize (``$*$'' means any value).}
\end{thm}
\begin{proof}
First we prove using Theorem~\ref{thm1} that the cases listed
in the table are not generically split. Indeed, assume ther contrary.
Then in cases 3) and 4)
the right hand side of formula~(\ref{form}) does not have a term of degree $1$.
On the other hand, the Picard group $\Pic(\BX)$ is rational.

In cases 1), 5), 7) the right hand side of formula~(\ref{form})
does not have a term of degree $3$. On the other hand, the group
$\Ch^3(\BX)$ is rational: in all these cases it is contained
in the subring generated by (rational) $\Ch^1(\BX)$, $\Ch^2(\BX)$,
and by the Chern classes of the tangent bundle $T_{\BX}$ (which are rational).

In case 8) one comes to the same contradiction considering $\Ch^4(\BX)$.
In case 2) one easiely comes to a contradiction, since in this case
the right hand side of formula~(\ref{form}) does not divide 
the Poincar\'e polynomial $g(\Ch^*(\BX),t)$.

Next we show that all other varieties not listed in the table
are generically split.
Let $G$ and $X$ be an exceptional group and
a $G$-variety not listed in the table. Consider $G_{k(X)}$.
Using Proposition~\ref{lemj}(2) one
immediatelly sees case-by-case that for all primes $p$
the $J$-invariant of the anisotropic kernel of
$G_{k(X)}$ is trivial. Therefore this anisotropic kernel is trivial
by Proposion~\ref{lemj}(4) and the group $G$ splits over $k(X)$.
\end{proof}

\begin{rem}
The cases 3) and 4) in the table above also follow from the
index reduction formula for exceptional groups \cite{MPW98}.
\end{rem}

\begin{rem}
In view of the results obtained in \cite{PSZ07} and in the present
paper the following holds:

\begin{enumerate}
\item $\Phi=\F,\E_6$, $J_2(G)=(1)$ if and only if $G$ has a non-trivial cohomological
invariant $f_3$.
\item $\Phi=\E_6$, $J_3(G)=(j_1,*)$
or $\Phi=\E_7$, $J_2(G)=(j_1,*,*,*)$, $j_1\ne 0$,
if and only if $G$ has a non-trivial Tits algebra.
\item $\Phi=\E_6,\E_7$, $J_3(G)=(1)$ if and only if $G$ has a non-trivial cohomological
invariant $g_3$.
\end{enumerate}
\end{rem}

\section{Index of groups of type $\E_7$}\label{sec6}
In this section we prove an index reduction formula for
groups of type $\E_7$. Our result can be considered as a
generalization of the usual index reduction formula
for central simple algebras.

Other variations on this theme are the Main Tool Lemma of A.~Vishik
\cite[Theorem~3.1]{Vi07} and an application of the
Rost degree formula \cite[Theorem~7.2]{Me03}.

To prove the main result of this section we use Chow motives
and Steenrod operation and a relation between rational cycles
on projective homogeneous varieties and their splitting properties.

To simplify the notation we will denote the Lefschetz motives in the category
of Chow motives with $\zz/p$-coefficient not as $(\zz/p)(i)$, but
still as $\zz(i)$.
The restriction on the characteristic in the following theorem comes
from Steenrod operations that we use in the proof,
since so far they are not constructed in characteristic $2$.

\begin{thm}\label{e7thm}
Let $G$ be an anisotropic group of type $\E_7$ and $X$ (resp. $Y$) be the
projective $G$-homogeneous variety of parabolic subgroups ot type
$1$ (resp. $7$) over a field $k$ with $\mathrm{char}\,k\ne 2$.
Then $Y$ has a $k(X)$-rational point
if and only if $Y$ has a zero-cycle of degree $1$.

In particular, $Y$ has no $k(X)$-rational points if $G$ has a non-trivial
Tits algebra, or if the absolute Galois group $\mathrm{Gal}(k_s/k)$
is a pro $p$-group, or if the Rost invariant of $G$ has order divisible by $4$,
or if $k$ is a perfect field with $\mathrm{char}\,k\ne 2,3$.
\end{thm}

\begin{proof}
Before proving this theorem we discuss briefly the plan of the proof.
First, assuming that $Y$ has a $k(X)$-rational point we find a rational
cycle, say $\alpha$, in codimension $17$ on the product $X\times X$. Using Steenrod operations
we produce certain projector-like cycle, say $\beta$, on $X\times X$.
Applying duality arguments to $\beta$ we find a sub-cycle, say $\gamma$, in the cycle
$\alpha$. On the other hand, multiplying $\alpha$ by a certain cycle known to be
always rational, we obtain again a projector-like cycle, but of another
shape because of the existence of $\gamma$ inside of $\alpha$. Applying
duality arguments again we come to a contradiction.

From now on an till the end of the proof of this theorem we set $p=2$.
To simplify the notation we denote as $\pt=[X_1]$ the class of a rational
point on $\BX$ (or $\overline{Y}$).

Assume first that the variety $Y$ does not have a zero-cycle of degree $1$,
but $Y$ has a $k(X)$-rational point.
All claims below are proved under this assumption.

\begin{clm}\label{clm3}
The varieties $X$ and $Y$ are not generically split.
\end{clm}
Assume that $J_2(G)$ is trivial. Then $G$ splits over an odd degree field
extension by \cite[Corollary~6.10]{PSZ07}. On the other hand, the variety
$Y$ becomes isotropic over a quadratic extension of $k$
by \cite[Corollary~3.4]{Fe72}. Therefore $Y$ has a zero-cycle of degree $1$.
Contradiction.

Thus, $J_2(G)$ is not trivial.
Now the variety $X$ is not generically split by Corollary~\ref{cor1}.
On the other hand, the variety $Y$ is
not generically split, since it has a rational point over $k(X)$
and $G$ is not split over $k(X)$.

\medskip
\medskip

In the following claims we use the Sweedler notation for Hopf
algebras to denote the cycles in the Chow rings of projective
homogeneous varieties, i.e., we do not write the sums and the indices.
\begin{clm}\label{clm4}
Some power of any element in a finite monoid is an idempotent.
In particular, for $q=1\times \pt+x_{(1)}\times x_{(2)}\in\Ch^{\dim X}(\BX\times\BX)$ with
$x_{(2)}\in\Ch^{<\dim X}(\BX)$
there exists $n\in\nn$ such that $q^{\circ n}$ is a non-trivial projector.
\end{clm}
The first statement of the claim is well known. For completeness we give
its proof.
Let $q$ be an element of our monoid. Since the monoid is finite,
we can find in the sequence $\{q^i\}_{i\in\nn}$ two equal cycles
$q^{n_1}$ and $q^{n_2}$ with $n_2\ge 2n_1$. Define $n=(n_2-n_1)n_1$.
Then $q^{2n}=q^{n}$, i.e., $q^{n}$ is an idempotent.

\begin{lem}
Let $X$ and $Y$ be arbitrary projective homogeneous varieties such that
$X_{k(Y)}$ and $Y_{k(X)}$ have zero-cycles of degree prime to $p$. Then 
the Chow motives $\Mot(X)$ and $\Mot(Y)$ with $\zz/p$-coefficients
have a common non-trivial
direct summand $R$ such that $R_{k_s}\simeq\zz\oplus M$
for some motive $M$.
\end{lem}
\begin{proof}
Since $X$ has a $k(Y)$-rational cycle of degree prime to $p$ and
$Y$ has a $k(X)$-rational cycle of degree prime to $p$, we can apply \cite[Lemma~1.6]{PSZ07} (generic point argument)
and get two cycles $\alpha\in\Ch^{\dim Y}(X\times Y)$ and
$\beta\in\Ch^{\dim X}(Y\times X)$ such that $\overline{\alpha}=
1\times \pt+x_{(1)}\times x_{(2)}$ and $\overline{\beta}=
1\times \pt+ x'_{(1)}\times x'_{(2)}$. The compositions
$\alpha\circ\beta$ and $\beta\circ\alpha$ give cycles on
$Y\times Y$ and $X\times X$ as in the previous claim. Therefore
some powers $(\alpha\circ\beta)^{\circ n}$ and $(\beta\circ\alpha)^{\circ n}$
define projectors over $k_s$. The mutually inverse isomorphisms
between the motives corresponding to these projectors are given
by the rational maps $\alpha$ and $\beta\circ(\alpha\circ\beta)^{\circ n-1}$.
Applying the Rost nilpotence theorem \cite[Section~8]{CGM05} we finish the proof
of the lemma.
\end{proof}

Now it easiely follows from the classification of Tits indices that
our varieties $X$ and $Y$ of the Theorem have a common motivic summand as above.
We denote this summand as $R$.

\begin{clm}\label{clm6}
$R_{k(X)}\simeq \zz\oplus\zz(17)\oplus R'$ for some motive $R'$.
\end{clm}
Consider the motive $R_{k(X)}$. We claim first that $R_{k(X)}\simeq
\zz\oplus\zz(l)\oplus R'$ for some motive $R'$ and some $l\in\zz$.
Indeed, let $q\in\Ch_{\dim X}(X\times X)$
be a projector corresponding to $R$. Consider
$\overline q$ over $k_s$. The cycle $\overline q\cdot\overline q^t$,
where $\overline q^t$ denotes the transposed cycle, equals
$n\pt\times\pt$, where $n$ is the dimension of the realization
$\Ch^*(X,\overline q)$. The number $n$ is, of course, even, since
otherwise we get a rational zero-cycle on $X$ of odd degree, and,
since by \cite[Corollary~3.4]{Fe72} the group $2\CH_0(\bar Y)$ is rational,
we get a zero-cycle of degree $1$ on $Y$ which contradicts the assumptions
of the Theorem. On the other hand,
the Krull-Schmidt theorem \cite[Theorem~9.6]{CM06} in the category of
Chow motives implies that $R_{k(X)}$ is a direct sum of the Lefschetz motives
and the indecomposable Rost motives. (The Rost motives are indecomposable,
since $(G_{k(X)})_{\an}$ has type $\D_4$ by our assumptions and
by Claim~\ref{clm3}). Therefore $R_{k(X)}$ must contain 
as a direct summand some Lefschetz motive $\zz(l)$ matching with $\zz$.

Next we compute $l$.
By assumptions the motive $R$ is a common motive of $X$ and $Y$.
Therefore the number $l$ has the property that $\zz(l)$ is a direct
summand of the motives $\Mot(X)$ and $\Mot(Y)$ over $k(X)$ (or over $k(Y)$).
Using the Hasse diagrams of $X$ and $Y$ (see \cite[Figures~21 and 23]{PlSeVa})
one can easiely see as in Example~\ref{exe7} that there is only
one such common dimension, namely $l=17$.

\medskip

In the following claim we put all computations that we need to prove
our theorem. The computations were done using algorithms described in
Section~\ref{sec2}. As $[i_1,\ldots,i_l]$ we denote the product
$s_{i_1}\ldots s_{i_l}$ in the Weyl group.
\begin{clm}\label{computer}
a) The $16$-th Steenrod operation (modulo $2$) of
$$f:=Z_{[7, 6, 5, 4, 3, 2, 4, 5, 6, 1, 3, 4, 5, 2, 4, 3, 1]}\in\Ch^{17}(\BX)$$
equals $S^{16}(f)=X_{[1]}=\pt$.

b) The $16$-th Chern class of the tangent bundle $T_{\BX}$ of $\BX$
equals 
\begin{align*}
c:=Z_{[6, 5, 4, 2, 3, 1, 4, 3, 5, 4, 2, 6, 5, 4, 3, 1]}
+Z_{[4, 2, 3, 1, 4, 3, 6, 5, 4, 2, 7, 6, 5, 4, 3, 1]}\\
+Z_{[4, 3, 1, 5, 4, 3, 6, 5, 4, 2, 7, 6, 5, 4, 3, 1]}\in\Ch^{16}(\BX).
\end{align*}

c) For any $g\in\Ch^9(\BX)$ and $h\in\Ch^8(\BX)\quad S^8(g)S^8(h)=cgh$.

d) $c\cdot\Ch^6(\BX)=c\cdot\Ch^{12}(\BX)=0$.

e) $S^8(\Ch^8(\BX))=(\zz/2)c$.
\end{clm}

{\it Proof of Theorem~\ref{e7thm}:}
Consider the cycle
$f=Z_{[7, 6, 5, 4, 3, 2, 4, 5, 6, 1, 3, 4, 5, 2, 4, 3, 1]}\in\Ch^{17}(\BX)$.
This cycle is rational over
$k(X)$ by \cite[Proposition~6.1]{CGM05}. Indeed, the Hasse diagram for $\BX$ is represented
in \cite[Figure~23]{PlSeVa} (That figure contains the left half
of the Hasse diagram which is too big to be represented
in whole. One should symmetrically reflect that diagram to get a complete
picture). Since by assumption the group $(G_{k(X)})_{\an}$ has type $\D_4$,
one should erase from the Hasse diagram of $\BX$ all edges
with labels $1$, $6$, and $7$. One immediately sees that the vertex
corresponding to
$f$ splits away from the diagram. Thus, $f$ is defined over $k(X)$.

By \cite[Lemma~1.6]{PSZ07} the cycle
$a:=1\times f+x_{(1)}\times x_{(2)}\in\Ch^{17}(\BX\times\BX)$,
$x_{(2)}\in\Ch^{<17}(\BX)$,
is rational (i.e., defined over $k$). By Claim~\ref{computer}(a) the cycle
$S^{16}(a)=
1\times \pt+x'_{(1)}\times x'_{(2)}\in\Ch^{33}(\BX\times\BX)$,
$x'_{(2)}\in\Ch^{<33}(\BX)$ ($\dim \BX=33$).
Therefore using Claim~\ref{clm4} one obtains a (rational)
projector on $\BX\times\BX$. We denote this projector as $q$.

We claim that this projector contains a summand of the form $r\times s$
with $r\in\Ch^{17}(\BX)$, $s\in\Ch^{16}(\BX)$ and $rs=\pt$.
Indeed, by Claim~\ref{clm6} the motive $\Mot(X)$ has an
indecomposable direct summand $R$ such the projector
corresponding to $R_{k_s}$ contains the sum $1\times \pt+r'\times s'$ with
$r'\in\Ch^{17}(\BX)$, $s'\in\Ch^{16}(\BX)$ and $r's'=\pt$.
Our projector $q$ contains
the summand $1\times \pt$. Therefore the Krull-Schmidt theorem for
Chow motives with finite coefficients \cite[Theorem~9.6]{CM06} implies that
$q$ must also contain a summand $r\times s$ with
$r\in\Ch^{17}(\BX)$, $s\in\Ch^{16}(\BX)$ and $rs=\pt$.

Since $q$ comes from the cycle $a$ by means of $S^{16}$,
the cycle $a$ contains a summand $g\times h$ with $g\in\Ch^9(\BX)$
and $h\in\Ch^8(\BX)$. Indeed, otherwise we can't get the cycle $r\times s$
in $q$, since $S^i(\alpha)=0$ for all $\alpha\in\Ch^{<i}(\BX)$. Moreover,
$a$ contains a summand $g\times h$ ($g\in\Ch^9(\BX)$, $h\in\Ch^8(\BX)$)
with $S^8(g)S^8(h)=\pt$, since $rs=\pt$.

Consider now the (rational)
product $(1\times c)\cdot a\in\Ch^{33}(\BX\times\BX)$.
This product contains the sum $1\times cf+g\times ch=1\times \pt+
g\times ch$. As in Claim~\ref{clm4} we may assume
that $(1\times c)\cdot a$ is a projector which contains
the summands $1\times \pt$ and $g\times ch$ by Claim~\ref{computer}(c).
The motive corresponding to this sum has the Poincar\'e polynomial
$1+t^9$. Consider now the motive of $X$ over $k(X)$. As in Example~\ref{exe7}
one can see that $\Mot(X)$ splits as a direct sum of (twisted) Lefschetz
motives and (twisted) Rost motives corresponding to the anisotropic kernel
of $G_{k(X)}$ which has strongly inner type $\D_4$ by our assumptions
and by Claim~\ref{clm3}.
The Poincar\'e polynomial of the Rost motive appearing in the motivic
decomposition is $1+t^3$. Moreover, the Lefschetz motive $\zz(9)$ does not
appear in the motivic decomposition. Therefore
by the Krull-Schmidt theorem \cite[Theorem~9.6]{CM06} the summand $g\times ch$
of the product $(1\times c)\cdot a$ is a part of a twisted Rost motive.
Therefore the product $(1\times c)\cdot a$ must contain a matching summand
for $g\times ch$. Since the Poincar\'e polynomial of our Rost
motive is $1+t^3$ the matching to $g\times ch$ summand has the form
$\widetilde g\times c\widetilde h$ with $\widetilde g\in\Ch^{9\pm 3}(\BX)$.
This leads to a contradiction with Claim~\ref{computer}(d).

Assume now that $Y$ does not have a $k(X)$-rational point. We will show that
the variety $Y$ does not have a zero-cycle of degree $1$.
We may assume that the Tits algebras of $G$ are trivial, since over each field
which makes $Y$ isotropic the Tits algebras of $G$ are split and, hence,
$Y$ does not have a zero-cycle of degree $1$.

Consider $Y_{k(X)}$.
Using the method of Chernousov-Gille-Merkurjev and
Brosnan (see Section~\ref{sec2}) one can decompose its motive as follows:
$\Mot(Y_{k(X)})\simeq\Mot(Q)\oplus\Mot(Z)(6)\oplus\Mot(Q)(17)$,
where $Q$ is the quadric corresponding to $G_{k(X)}$
and $Z$ is its
maximal orthogonal Grassmannian. In particular, the image of the
degree map $\deg\colon\CH_0(Y_{k(X)})\to\zz$ coincides with
the image of $\deg\colon\CH_0(Q)\to\zz$ which is known to be
$2\zz$ by Springer's theorem. Thus, $Y_{k(X)}$ does not have a zero-cycle
of degree $1$ and therefore $Y$ also doesn't. The theorem is proved.

The variety $Y$ does not have a zero-cycle of degree $1$ in the cases
listed in the Theorem either by obvious reasons, or by \cite{GS}.
\end{proof}

\begin{rem}
The index reduction formula for central simple algebras
\cite[Section~8, Type $\E_7$]{MPW98} implies
that Theorem~\ref{e7thm} holds if the group $G$ has a non-trivial Tits algebra.
Our proof does not depend on the (non)-triviality of the Tits algebras.
\end{rem}

\begin{rem}
The following papers discuss the problem of the existence of zero-cycles
of degree $1$ on anisotropic varieties: \cite{Fl04}, \cite{Par05},
\cite{Serre}, \cite{To04}.
\end{rem}

Now we summarize the results of the present paper and
provide a complete list of the the higher Tits indices
for anisotropic groups of type $\F$ and $^1{\E_6}$.

\medskip

\begin{tabular}{|r|r|r|}
\hline\label{tab}
$\F$ & $J_2=(0)$ & $\{1,\F\}$\\
\hline
$\F$ & $J_2=(1)$ & $\{1,\B_3,\F\}$\\
\hline
$\E_6$ & $J_2=(0)$, $J_3=(0,1)$ & $\{1,\E_6\}$\\
\hline
$\E_6$ & $J_2=(1)$, $J_3=(0,*)$ & $\{1,\D_4,\E_6\}$\\
\hline
$\E_6$ & $J_2=(0)$, $J_3=(j_1,*)$, $j_1\ne 0$ & $\{1,2\A_2,\E_6\}$\\
\hline
$\E_6$ & $J_2=(1)$, $J_3=(j_1,*)$, $j_1\ne 0$ & $\{1,2\A_2,\D_4,\E_6\}$\\
\hline
\end{tabular}

\medskip

Let now $G$ be an anisotropic group of type $\E_7$ with trivial
Tits algebras over a field $k$.
Denote as $\Omega(G)$ the higher Tits index of $G$.
Theorems~\ref{cor1} and \ref{e7thm} imply that

\noindent
$\D_4\in\Omega(G)$ iff $J_2(G)$ is non-trivial,

\noindent
$\E_6\in\Omega(G)$ iff $J_3(G)$ is non-trivial, and

\noindent
if $\mathrm{char}\,k\ne 2$, then
$\D_6\in\Omega(G)$ iff the variety of parabolic subgroups of $G$
of type $7$ does not have a zero-cycle of degree $1$.

In partucular, using \cite[Corollary~2.2]{GS} we have:
\begin{prop}\label{d6}
Let $k$ be a perfect field and $\mathrm{char}\,k\ne 2,3$.
Then $\Omega(G)\supset\{1,\D_4,\D_6,\E_7\}$. In particular, there exists
a field extension $K/k$ such that $G_K$ has semisimple anisotropic kernel
of type $\D_6$.
\end{prop}

The following corollaries of this statement
and their proofs are due to S.~Garibaldi.

\begin{cor}\label{coro1}
Let $G$ be a group of type $\E_7$ with trivial Tits algebras
over a perfect field with $\mathrm{char}\,k\ne 2,3$.
Then we can speak about {\it the Rost invariant} $r(G)$ of $G$
(see \cite[\S~31.B]{Inv}).
Assume that $r(G)$ lies in $H^3(k,\zz/4)$ and is a symbol.
Then $G$ is isotropic.
\end{cor}
\begin{proof}
For sake of contradiction, suppose $G$ is anisotropic. Take $K$ as
in Proposition~\ref{d6}. By assumptions the group $G\otimes_kK$ has
semisimple anisotropic kernel $\mathrm{Spin}(q)$ for a $12$-dimensional
quadratic form $q$ with the Arason invariant $e_3(q)$ a symbol.
By \cite[Lemma~12.5]{Ga07} the form $q$ is isotropic.
\end{proof}

\begin{cor}[Kernel of the Rost invariant]\label{coro2}
Let $G_0$ denote a split simply connected group of type $\E_7$ over
an arbitrary field $k$.
Suppose $\eta\in H^1(k,G_0)$ and the Rost invariant of $\eta$
is trivial. Then $\eta=0$.
\end{cor}
\begin{proof}
By \cite{Gi00} we may assume that $\mathrm{char}\,k=0$.
By Corollary~\ref{coro1} the twisted group $_\eta{(G_0)}$ is isotropic.
So $\eta$ is equivalent to a cocycle with values in $\E_6^{sc}$
(a split simply connected group of type $\E_6$) or $\mathrm{Spin}_{12}(q)$
($q$ split).
But the triviality of the kernel of the Rost invariant for these two groups
is well known (see \cite[Theorem~0.5]{Ga01b}).
\end{proof}

Thus, we provided a shortened proof of the triviality of the kernel of the
Rost invariant for groups of type $\E_7$; cf. \cite{Ga01b} and \cite{C03}. 

\bibliographystyle{chicago}

\newpage
\begin{center}
{\large \bf Appendix}

by Mathieu Florence
\end{center}

We prove here the following result:
\begin{thm}\label{florence}
Let $\Phi$ be a root system.
There exists a field $K$ and a group $G$ over $K$  of type $\Phi$
such that the higher Tits index of $G$ is the maximal possible, i.e.,
$$\{\ind (G_L)_{\an}\mid L/K\text{ a field extension}\}$$
$$=\{\ind (H_l)_{\an}\mid l\text{ a field}, H\text{ a group over }l
\text{ of type }\Phi\}.$$
\end{thm}
\begin{proof}
Let $k$ be a prime field and $G_0$ a split semisimple simply connected group
of type $\Phi$ over $k$.
Its automorphism group $H=\Aut(G_0)$ is an algebraic group. Consider
a generically free finite-dimensional representation $V$ of $H$ over $k$.
We can find an open
subset $U\subset\mathbb{A}(V)$ of the affine space of $V$ which is
$H$-invariant such that the categorical quotient $U\to U/H$ exists and is an
$H$-torsor. Let $\eta$ be the generic point of $U/H$ and $(G_0)_\eta$
(resp. $H_\eta$, $P$) denote the
pullback $G_0\times_{\Spec k} \eta$ (resp. $H\times_{\Spec k} \eta$,
$U\times_{U/H} \eta$). We view $(G_0)_\eta$ and $H_\eta$ as algebraic
groups over $\eta$, and $P$ as a $H_\eta$-torsor.
Denote by $G$ the twist of $(G_0)_\eta$ by $P$. For a standard
parabolic subset $\Psi \subset \Phi$, let $X_\Psi \longrightarrow \eta$
be the homogenous space of parabolic subgroups of $G$ of type $\Psi$.
We claim that there exists a $k$-variety $S$, with the function field $\eta$,
a faithfully flat group scheme $\mathfrak G \longrightarrow S$ and
$\mathfrak G$-homogeneous spaces $\mathfrak X_ \Psi \longrightarrow S$
which are models of $G$ and of the $X_\Psi$'s respectively, and which are
generic in the following sense:

\medskip

\noindent
{\bf (*)} For every field extension $K/k$ with $K$ infinite, and every semisimple
simply connected group $G'$ of type $\Phi$ over $K$, there exists a point
$x \in S(K)$ such that the pullback of $\mathfrak G \longrightarrow S$ by
$x$ is isomorphic to $G'$, and such that, for all standard parabolic subsets
$\Psi \subset \Phi$, the homogeneous space of parabolic subroups of type
$\Psi$ of $G'$ is isomorphic to the pullback of
$\mathfrak X_\Psi \longrightarrow S$ by $x$.

\medskip

Moreover, the same remains true if we replace $S$ by a nonempty open suvariety
$S' \subset S$. If we remove the assertion concerning standard parabolic
subgroups from (*), then the result is well-known\footnote
{G.~Berhuy, G.~Favi, Essential dimension: a functorial point of view,
{\it Doc. Math.} {\bf 8} (2003), 279--330, Proposition 4.11.}. The proof of (*) is then a formal
consequence of the following facts: \\
{\bf (i)} For $K$, $G'$ and $\Psi$ as in (*), the homogeneous space of parabolic
subgroups of type $\Psi$ of $G'$ is isomorphic to the twist of the homogeneous
space of parabolic subgroups of type $\Psi$ of $G_0 \times_{\Spec k} \Spec K$
by the $H$-torsor (defined over $K$) corresponding to $G'$;\\
{\bf (ii)} Let $Y$ denote any $H$-variety (over $k$), and $P$ be an $H$-torsor
over a base $T/ k$ which is integral and of finite-type. Let $l$ denote the field
of functions of $T$. Then the twist of $Y \times_{\Spec k} \Spec l$ by
$P \times_T \Spec l$ can already be defined over a nonempty open subvariety $S$
of $T$.

We omit the proofs of these assertions (for the second one, one has to use
standard finite-typeness arguments).

Now let $K/k$ be a field extension, with $K$ infinite, and $G'$ a semisimple
simply connected group of type $\Phi$, with anisotropic kernel of type
$\Psi \subset \Phi$. Let $\eta_\Psi$ denote the generic point of $X_\Psi$.
We claim that $G_\Psi:=G \times_\eta \eta_\Psi$ has anisotropic kernel of type
$\Psi$. Indeed, assume the contrary. Then $G_\Psi$ has a strictly smaller
type $\Psi' \subset \Psi$. Thus there exists a rational map
$X_\Psi \dasharrow X_{\Psi'}$ defined over $\eta$. Shrinking $S$, we may assume
that there exists an $S$-morphism
$\mathfrak X_\Psi \longrightarrow \mathfrak X_{\Psi'} $. Now let $x \in S(K)$
be a rational point as in (*). Specializing at $x$, we obtain that $G'$ ---
which possesses parabolic subgroups of type $\Psi$ by assumption --- also
possesses parabolic subgroups of type $\Psi'$, a contradiction.
\end{proof}

\noindent
{\sc
V.~PETROV\\
PIMS\\
University of Alberta\\
Department of Mathematical and Statistical Sciences\\
Edmonton, AB T6G 2G1\\
Canada}

\medskip

\noindent
{\sc N.~SEMENOV\\
Mathematisches Institut\\
Universit\"at M\"unchen\\
Theresienstr. 39\\
80333 M\"unchen\\
Germany}

\medskip

\noindent
{\sc M.~FLORENCE\\
Universit\'e Paris 6\\
175, rue du Chevaleret\\
75013 Paris\\
France}

\end{document}